\documentclass[reqno]{amsart}
\usepackage{mathrsfs}
\usepackage{times}
\usepackage[T1]{fontenc}
\usepackage{mathrsfs}
\usepackage{latexsym}
\usepackage[dvips]{graphics}
\usepackage{epsfig}
\usepackage{amsmath,amsfonts,amsthm,amssymb,amscd}
\input amssym.def
\input amssym.tex
\usepackage{color}
\usepackage[all]{xypic}

\newcommand{\LV}{\left|}
\newcommand{\RV}{\right|}
\newcommand{\LC}{\left(}
\newcommand{\RC}{\right)}
\newcommand{\LB}{\left[}
\newcommand{\RB}{\right]}
\newcommand{\LCB}{\left\{}
\newcommand{\RCB}{\right\}}

\newcommand{\Int}{\int_{\mathbb{R}}}

\newtheorem{theorem}{Theorem}[section]
\newtheorem{lemma}[theorem]{Lemma}
\newtheorem{proposition}[theorem]{Proposition}

\newtheorem{definition}[theorem]{Definition}
\theoremstyle{remark}
\newtheorem{remark}[theorem]{Remark}
\newtheorem*{remark*}{Remark}
\theoremstyle{remark}
\begin{document}

\title{Stability of solitary waves and global existence of a generalized two-component Camassa-Holm system}
\author{Robin Ming Chen}
\address{Robin Ming Chen\newline
School of Mathematics\\
University of Minnesota\\
Minneapolis, MN 55455} \email{chenm@math.umn.edu}
\author{Yue~Liu}
\address{Yue Liu\newline
University of Texas at Arlington, Department of Mathematics, Arlington, TX 76019-0408}
\email{yliu@uta.edu}
\author{Zhijun Qiao}
\address{Zhijun Qiao\newline
Department of Mathematics, University of Texas-Pan American, Edinburg, TX 78539}
\email{qiao@utpa.edu}

\thanks{The work of R.M. Chen was partially supported by the NSF grant DMS-0908663. The work of Y. Liu was partially supported by the NSF grant DMS-0906099 and the NHARP grant 003599-0001-2009. The work of Z. Qiao was partially supported by the NHARP grant 003599-0001-2009 and the USARO grant W911NF-08-1-0511.}

\maketitle \numberwithin{equation}{section}
\begin{abstract}
 We study here the existence of solitary wave solutions of a generalized two-component
Camassa-Holm system. In addition to those smooth solitary-wave solutions, we show that
there are  solitary waves with singularities: peaked and cusped solitary waves. We also
demonstrate that  all smooth solitary waves are orbitally stable in the energy space.
We finally give a sufficient condition for global strong solutions to the equation
in some special case.
\end{abstract}


\section{Introduction}
There are several classical models describing the motion of
waves at the free surface of shallow water under the influence
of gravity. Among these models, the best known is the
Korteweg-de Vries (KdV) equation \cite{Na,Wh}
\begin{equation*}
u_t+6uu_x+u_{xxx}=0.
\end{equation*}
The KdV equation admits solitary wave solutions, i.e. solutions
of the form $u(t,x)=\varphi(x-ct)$ which travel with fixed speed $c$,
and that vanish at infinity. The KdV solitary waves are smooth and
retain their individuality under interaction and eventually emerge
with their original shapes and speeds \cite{Dr}. Moreover, KdV is
an integrable infinite-dimensional Hamiltonian system \cite{Mc}.
However, the KdV equation does not model the phenomenon of breaking
for water waves. Instead, as soon as the initial profile
$u_0\in H^1(\mathbb{R})$, the solutions are global in time \cite{KPV},
whereas some shallow water waves break \cite{Wh}.

Another model, the Camassa-Holm (CH) equation \cite{CH1}
\begin{equation}\label{CH}
u_t-u_{xxt}+3uu_x=2u_xu_{xx}+uu_{xxx}
\end{equation}
arises as a model for the unidirectional propagation of shallow
water waves over a flat bottom \cite{CH1,CL,DGH1,DGH2,DGH3,Jo1}, as well as water
waves moving over an underlying shear flow \cite{Jo2}. Equation
\eqref{CH} is completely integrable with the Lax pair \cite{CH1} and
with infinitely many conservation laws as an bi-Hamiltonian system
\cite{FF}.

The CH equation has many remarkable properties that KdV does not
have like solitary waves with singularities and breaking waves. The
CH equation admits peaked solitary waves or ``peakons''
\cite{BSS,CH1,Le}: $u(t,x)=ce^{-|x-ct|}$, $c\neq0$, which are smooth
except at the crests, where they are continuous, but have a jump
discontinuity in the first derivative. The peakons capture a
feature that is characteristic for the waves of great height --
waves of the largest amplitude that are exact solutions of the
governing equations for water waves \cite{Co4,CM,To}. The CH
equation also models wave breaking (i.e. the solution remains
bounded while its slope becomes unbounded in finite time)
\cite{CH1,Co1,Co2,Co3,CE1,Mc,Wh}.

The CH equation also admits many multi-component generalizations. It
is intriguing to know if the above two properties may persist in the
systems. In this paper we consider the following generalized
two-component CH system established in \cite{CLi} which can be
derived from shallow water theory with nonzero constant vorticity
\begin{equation}\label{genCH2_uandrho}
\left\{\begin{array}{l}
u_t-u_{txx}-Au_x+3uu_x-\sigma(2u_xu_{xx}+uu_{xxx})+\rho\rho_x=0,\\
\rho_t+(\rho u)_x=0,
\end{array}\right.
\end{equation}
or equivalently, using the linear momentum $m=u-u_{xx}$,
 \begin{equation*}
\left\{\begin{array}{l}
m_t+\sigma um_x -Au_x +2\sigma mu_x+3(1-\sigma)uu_x+\rho\rho_x=0,\\
\rho_t+(\rho u)_x=0,
\end{array}\right.
\end{equation*}
where $u(t,x)$ is the horizontal velocity and $\rho(t,x)$ is related
to the free surface elevation from equilibrium (or scalar density)
with the boundary assumptions $u\to0$, $\rho\to 1$ as
$|x|\to\infty$. The scalar $A>0$ characterizes a linear underlying
shear flow and hence system \eqref{genCH2_uandrho} models
wave-current interactions. The real dimensionless constant $\sigma$ is a parameter which provides the competition, or balance, in fluid convection between nonlinear steepening and amplification due to stretching.

When $\sigma=1$ it recovers the standard two-component CH system
which is completely integrable \cite{CI, Iva, SA} as it can be
written as compatibility condition of two linear systems (Lax pair)
with a spectral parameter $\zeta$, that is,
\begin{align*}
\Psi_{xx}&=\LB -\zeta^2\rho^2+\zeta\LC m-{A\over2} +{1\over4} \RC
\RB \Psi,\\
\Psi_t&=\LC {1\over 2\zeta}-u \RC\Psi_x+{1\over2}u_x\Psi.
\end{align*}
In the case $\rho\equiv0$, it becomes
\begin{equation}\label{eqnrod}
 u_t-u_{xxt}+3uu_x=\sigma\LC 2u_xu_{xx}+uu_{xxx} \RC,
\end{equation}
which models finite length, small amplitude radial deformation waves in cylindrical
hyperelastic rods \cite{Dai}. System \eqref{genCH2_uandrho} has the following two Hamiltonians
\begin{align*}
 &H_1=\frac12\Int\ \left(u^2+u^2_x+(\rho-1)^2\right)\ dx,\\
 &H_2=\frac12\Int\ \left(u^3+\sigma uu^2_x+2u(\rho-1)+u(\rho-1)^2-Au^2\right)\ dx.
\end{align*}

We study solitary wave solutions of \eqref{genCH2_uandrho}, i.e. solutions of the form
\begin{equation*}
(u(x,t), \rho(x,t))=(\varphi(x-ct), \rho(x-ct)), \quad c\in\mathbb{R}
\end{equation*}
for some $\varphi, \rho: \mathbb{R}\to \mathbb{R}$ such that
$\varphi\to0$, $\rho\to1$ as $|x|\to\infty$. In the study of the CH
traveling waves it was observed through phase-plane analysis
\cite{LO} that both peaked and cusped traveling waves exist.
Subsequently, Lenells \cite{Le, Le0} used a suitable framework for
weak solutions to classify all weak traveling waves of the CH
equation \eqref{CH} and the hyperelastic rod equation
\eqref{eqnrod}.

Using a natural weak formulation of the two-component CH system, we
will establish exactly in what sense the peaked and cusped solitary
waves are solutions. It was shown in \cite{CI,LZ, Mu} that when
$\sigma=1$ the two component system \eqref{genCH2_uandrho} has only
smooth solitary waves, with a single crest profile and exponential
decay far out. In \cite{HNT}, the authors considered a modified two-component CH equation which
allows dependence on average density as well as pointwise density and a
linear dispersion is added to the first equation of the system. They showed that
the modified system admits peakon solutions in both $u$ and $\rho$. However it
is unclear whether the generalized two-component CH system \eqref{genCH2_uandrho} has
solitary waves with singularities.
We show here peaked solitary waves exist when
$\sigma>1$. We also provide an implicit formula for the peaked
solitary waves. However whether these peaked solitary waves are
solitons still remains open.

The stability of solitary waves is one of the fundamental
qualitative properties of the solutions of nonlinear wave equations
\cite{St}. Due to the fact that the solitons hardly interact with
each other at all it is reasonable to expect that they are stable.
It has been proved that for the CH equation, the smooth solitary
waves are orbitally stable \cite{CW2}. Moreover, the peakons,
whether solitary waves or periodic waves, are also orbitally stable
\cite{CMo, CW1, Le1, Le2}. It was shown in \cite{LZ} that when
$\sigma=1$ all solitary waves are orbitally stable. We prove in this
paper that when $\sigma\leq1$ all smooth solitary waves are
orbitally stable. The proof of the stability basically follows the
general approach in \cite{GSS}. In comparison with the spectral
arguments on the Hessian operator in \cite{GSS}, here we require more
precise analysis on the spectrum of a linearized operator around the
solitary waves for the system \eqref{genCH2_uandrho}.

A special case of system \eqref{genCH2_uandrho} is when $\sigma=0$.
In the scalar equation case when $\sigma=0$ it is the BBM equation \cite{BBM}.
The solutions are shown to be global in time. 
We show that the same results hold in the system case.

This paper is organized as follows. In Section \ref{sec_solwaves} we classify
the solitary waves of \eqref{genCH2_uandrho}. In particular we show the existence
of peaked solitary waves for $\sigma>1$. In Section \ref{sec_stab} we prove that
when $\sigma\leq1$ all smooth solitary waves are nonlinearly stable. Finally in
Section \ref{sec_global},  we show that the system \eqref{genCH2_uandrho} is
globally well-posed for $\sigma=0$.

\vspace{.1in}


\section{Solitary waves}\label{sec_solwaves}


Let $X=H^1(\mathbb{R})\times L^2(\mathbb{R})$ be a real Hilbert space with inner product $(,)$, and
denote its element by $\vec{u}=(u,\eta)$. The dual of $X$ is
$X^*=H^{-1}(\mathbb{R})\times L^{2}(\mathbb{R})$ and a natural isomorphism $I$ from $X$
to $X^*$ can be
defined by
\begin{equation*}
 I=\LC \begin{array}{cc}
        1-\partial_x^2 & 0\\
        0              & 1
       \end{array}
 \RC.
\end{equation*}
Using the map $I$, the paring $\langle,\rangle$ between $X$ and $X^*$ can be represented as
\begin{equation*}
 \langle I\vec{u}, \vec{v} \rangle = \langle u, v \rangle_1+ \langle \eta, \xi \rangle_0, \quad \hbox{for }\
\vec{u}=(u,\eta)\in X,\ \vec{v}=(v,\xi)\in X^*,
\end{equation*}
where $\langle , \rangle_s$ denotes the $H^s(\mathbb{R})\times H^{-s}(\mathbb{R})$ dual pairing.
We will identify the second dual $X^{**}$ with $X$ in a natural way.

Since $\rho\to 1$ as $|x|\to\infty$ in \eqref{genCH2_uandrho}, we can let $\rho=1+\eta$ with $\eta\to0$
as $|x|\to\infty$ and hence we can rewrite system \eqref{genCH2_uandrho} as
\begin{equation}\label{genCH2_uandeta}
\left\{\begin{array}{l}
u_t-u_{txx}-Au_x+3uu_x-\sigma(2u_xu_{xx}+uu_{xxx})+(1+\eta)\eta_x=0,\\
\eta_t+\LC (1+\eta) u \RC_x=0.
\end{array}\right.
\end{equation}
The two Hamiltonians introduced in the Introduction define the following two functionals on $X$
\begin{align}
 & E(\vec{u})={1\over2}\Int\ \left(u^2+u^2_x+\eta^2\right)\ dx, \label{consE}\\
 & F(\vec{u})={1\over2}\Int\ \left(u^3+\sigma uu^2_x+2u\eta+u\eta^2-Au^2\right)\ dx, \label{consF}
\end{align}
with $\vec{u}=(u,\eta)\in X$. The quantity $E(\vec{u})$ associates with the translation invariance of
\eqref{genCH2_uandeta}. Using functional $F(\vec{u})$, system \eqref{genCH2_uandeta} can be written
in an abstract Hamiltonian form
\begin{equation}\label{absF}
 \vec{u}_t = J F'(\vec{u}),
\end{equation}
where $J$ is a closed skew symmetric operator given by
\begin{equation*}
 J=\LC \begin{array}{cc}
        -\partial_x(1-\partial_x^2)^{-1} & 0\\
        0              & -\partial_x
       \end{array}
 \RC 
\end{equation*}
and $F'(\vec{u}): X\to X^*$  is the variational derivative of $F$ in $X$ at $\vec{u}$.

Note that if
\begin{equation}\label{def_p}
 p(x):=\frac12e^{-|x|}, x\in \mathbb{R},
\end{equation}
then
$(1-\partial^2_x)^{-1}f=p\ast f$ for all $f\in L^2(\mathbb{R})$.  We can then further rewrite system
\eqref{genCH2_uandeta} in a weak form as
\begin{equation}\label{genCH2_uandeta_convolution}
\left\{\begin{array}{l}
\displaystyle u_t+\sigma uu_x+\partial_x
p\ast\LC -Au+\frac{3-\sigma}{2}u^2+\frac{\sigma}{2}u_x^2+\frac12(1+\eta)^2 \RC=0,\\\\
\eta_t+\LC (1+\eta)u \RC_x=0.
\end{array}
\right.
\end{equation}

\begin{definition}\label{def_solution}
 Let $0<T\leq\infty$. A function $\vec{u}=(u,\eta)\in C\LC [0,T); X \RC$
is called a solution
of \eqref{genCH2_uandeta} on $[0,T)$ if it satisfies \eqref{genCH2_uandeta_convolution} in
the distribution
sense on $[0,T)$ and $E(\vec{u})$ and $F(\vec{u})$ are conserved.
\end{definition}

Applying transport equation theory combined with the method of Besov spaces, one may follow
the similar argument as in \cite{GL1} to obtain the following local well-posedness result for the system
\eqref{genCH2_uandeta}.
\begin{theorem}\label{thm_localwellposedness}
 If $(u_0, \eta_0)\in H^s\times H^{s-1}$, $s>3/2$, then there exist
 a maximal time $T=T(\|(u_0,\eta_0)\|_{H^s\times H^{s-1}})>0$ and a
 unique solution $(u,\eta)$ of \eqref{genCH2_uandrho} in
 $C([0,T); H^s\times H^{s-1})\cap C^1([0,T); H^{s-1}\times H^{s-2})$
 with $(u,\eta)|_{t=0}=(u_0,\eta_0)$. Moreover, the solution depends
 continuously on the initial data and $T$ is independent of $s$.
\end{theorem}

It is easily seen from the embedding $H^1(\mathbb{R})\hookrightarrow L^\infty(\mathbb{R})$ that
$E(\vec{u})$ and $F(\vec{u})$ are both well defined in $H^s\times H^{s-1}$ with $s>3/2$, and $E(\vec{u})$
is conserved, as suggested
in the local wellposedness Thoerem \ref{thm_localwellposedness}. From \eqref{absF} we see that
\begin{equation*}
 {d\over dt}F(\vec{u})=\langle F'(\vec{u}), \vec{u}_t \rangle = \langle F'(\vec{u}), JF'(\vec{u}) \rangle
=0.
\end{equation*}
So $F(\vec{u})$ is also invariant.

Now we give the definition of {\em solitary waves} of \eqref{genCH2_uandeta}.
\begin{definition}\label{def_solitary}
 A solitary wave of \eqref{genCH2_uandeta} is a nontrivial traveling wave solution of
\eqref{genCH2_uandeta}
of the form $\vec{\varphi}_c(t,x)=\LC \varphi_c(x-ct), \eta_c(x-ct) \RC\in H^1(\mathbb{R})\times H^1(\mathbb{R})$
with $c\in \mathbb{R}$ and
$\varphi_c, \eta_c$ vanishing at infinity.
\end{definition}

For a solitary wave $\vec{\varphi}=(\varphi,\eta)$ with speed $c\in \mathbb{R}$, it satisfies
\begin{equation}\label{soleqnfull}
 \left\{\begin{array}{l}
\displaystyle         \LB -c\varphi+{\sigma\over2} \varphi^2+
p\ast\LC -A\varphi+\frac{3-\sigma}{2}\varphi^2+\frac{\sigma}{2}\varphi_x^2+\frac12(1+\eta)^2 \RC \RB_x=0,\\\\
         \LB-c\eta+(1+\eta)\varphi \RB_x=0,
        \end{array}
 \right. \quad \hbox{in } \mathcal{D}'(\mathbb{R}).
\end{equation}
Integrating the above system and applying $(1-\partial^2_x)$ to the first equation we get
\begin{equation}\label{soleqn}
 \left\{\begin{array}{l}
\displaystyle  -(c+A)\varphi+c\varphi_{xx}+{3\over2} \varphi^2=\sigma \varphi\varphi_{xx}+
           \frac{\sigma}{2}\varphi_x^2-\frac12(1+\eta)^2 +{1\over2},
\quad \hbox{in } \mathcal{D}'(\mathbb{R}).\\
          -c\eta+(1+\eta)\varphi =0.
        \end{array}
 \right.
\end{equation}
The fact that the second equation of the above holds in a strong sense comes from the regularity of
$\varphi$ and $\eta$.

\begin{proposition}\label{prop_sol}
 If $(\varphi, \eta)$ is a solitary wave of \eqref{genCH2_uandeta} for some
$c\in \mathbb{R}$, then $c\neq0$ and $\varphi(x)\neq c $ for any $x\in\mathbb{R}$.
\end{proposition}
\begin{proof}
 From the definition of solitary waves and the
embedding theorem we know that $\varphi$ and
$\eta$ are both continuous. If $c=0$, then \eqref{soleqn} becomes
\begin{equation}\label{soleqnc=0}
 \left\{\begin{array}{l}
\displaystyle  -A\varphi+{3\over2} \varphi^2=\sigma \varphi\varphi_{xx}+
           \frac{\sigma}{2}\varphi_x^2-\frac12(1+\eta)^2 +{1\over2},\\
          (1+\eta)\varphi =0.
        \end{array}
 \right.
\end{equation}
Since $\eta$ vanishes at infinity, the second equation of the above system indicates that
$\varphi(x)=0$ for $|x|$ large enough. Denote $x_0=\max\{x:\ \varphi(x)\neq0 \}$. Hence
$\varphi(x)=0$ on $[x_0,\infty)$ and $\varphi \not \equiv 0$ on $(x_0-\delta, x_0)$ for
any $\delta>0$. Consider now the first equation of \eqref{soleqnc=0} on $[x_0,\infty)$ we see
that $\eta\equiv0$ on $[x_0,\infty)$. Then the continuity of $\eta$ implies that there exists a
$\delta_1>0$ such that $1+\eta(x)>0$ on $(x_0-\delta_1, x_0)$. This together with the second
equation of \eqref{soleqnc=0} leads to $\varphi(x)\equiv0$ on $(x_0-\delta_1, x_0)$, which
is a contradiction. Therefore $c\neq0$.

Next we show $\varphi\neq c$. If not and there is some $x_1\in\mathbb{R}$ such that
$\varphi(x_1)=c$. Then the second equation of \eqref{soleqn} infers that
\begin{equation*}
 \varphi(x_1)=\LC c-\varphi(x_1) \RC \eta(x_1) = 0,
\end{equation*}
so $c=0$, which is a contraction.
\end{proof}

Using the above proposition we obtain from the second equation of \eqref{soleqn} that
\begin{equation}\label{soleta_phi}
 \eta={\varphi \over c-\varphi}.
\end{equation}
Plugging this into the first equation of \eqref{soleqn} we obtain an equation for the unknown
$\varphi$ only
\begin{equation}\label{eqnphi}
 -(c+A)\varphi+c\varphi_{xx}+{3\over2} \varphi^2=\sigma \varphi\varphi_{xx}+
           \frac{\sigma}{2}\varphi_x^2-\frac12{c^2\over (c-\varphi)^2} +{1\over2},
\quad \hbox{in } \mathcal{D}'(\mathbb{R}).
\end{equation}

\subsection{The case when $\sigma=0$}

When $\sigma=0$, \eqref{eqnphi} becomes
\begin{equation}\label{eqnphi_sigma=0}
 \varphi_{xx}={c+A\over c}\varphi-{3\over2c} \varphi^2+{1\over2c}-
\frac12{c\over (c-\varphi)^2},
\quad \hbox{in } \mathcal{D}'(\mathbb{R}).
\end{equation}
Since $\varphi\in H^1(\mathbb{R})$ and $c-\varphi\neq0$ we know that $|c-\varphi|$ is bounded away
from 0. Hence from the standard local regularity theory to elliptic equation we see
that $\varphi\in C^\infty(\mathbb{R})$ and so is $\eta$. Therefore in this case all solitary waves are smooth.

As for the existence, we may multiply \eqref{eqnphi_sigma=0} by $\phi_x$ and integrate on $(-\infty,x]$ to get
\begin{equation}\label{eqnphi_x_sigma=0}
 \varphi^2_x={\varphi^2(c-\varphi-A_1)(c-\varphi-A_2) \over c(c-\varphi)}
 := G(\varphi),
\end{equation}
where
\begin{equation}\label{def_A_12}
 A_1={-A+\sqrt{A^2+4} \over 2}, \quad A_2={-A-\sqrt{A^2+4} \over 2}
\end{equation}
are the two roots of the equation $y^2+Ay-1=0$. Since $A>0$, we know $A_1>0>A_2$.

From the decay property of $\phi$ at infinity we know that a necessary condition for the existence is $c\geq A_1$ or $c\leq A_2$. But one may prove further that
\begin{theorem}\label{thm_exist_sigma=0}
When $\sigma=0$, \eqref{genCH2_uandeta} admits a solitary wave solution if and only if
\begin{equation}\label{cond_c}
 c>A_1 \quad \hbox{or} \quad c<A_2.
\end{equation}
All solitary waves are smooth in this case.
\end{theorem}
\begin{proof}
The regularity is discussed as above. So we will just focus on the existence part.

If $c=A_1$, then \eqref{eqnphi_x_sigma=0} becomes
\begin{equation}\label{eqnphi_x-A_1_sigma=0}
 \varphi^2_x={-\varphi^3(A_1-A_2-\varphi) \over A_1(A_1-\varphi)}:=G_1(\varphi).
\end{equation}
Hence we see that $\varphi(x)<0$ near $-\infty$. Because $\varphi(x)\to0$ as $x\to-\infty$, there is some $x_0$
sufficiently large negative so that
$\varphi(x_0)=-\epsilon<0$, with $\epsilon$ sufficiently small, and $\varphi_x(x_0)<0$. From standard ODE theory, we can
generate a unique local solution $\varphi(x)$ on $[x_0-L, x_0+L]$ for some $L>0$. Since $A_1>0>A_2$, we have
\begin{equation}\label{eqn_decreasing}
 \LB {-\varphi^3(A_1-A_2-\varphi) \over (A_1-\varphi)} \RB'={\varphi^2\LB -3\varphi^2+(6A_1-2A_2)\varphi
-3A_1(A_1-A_2) \RB\over (A_1-\varphi)^2}<0,
\end{equation}
for $\varphi<0$. Therefore $G_1(\varphi)$ decreases for $\varphi<0$.
Because $\varphi_x(x_0)<0$, $\varphi$ decreases near $x_0$, so
$G_1(\varphi)$ increases near $x_0$. Hence from
\eqref{eqnphi_x-A_1_sigma=0}, $\varphi_x$ decreases near $x_0$, and then
$\varphi$ and $\varphi_x$ both decreases on $[x_0-L, x_0+L]$. Since
$\sqrt{G_1(\varphi)}$ is locally Lipschitz in $\varphi$ for
$\varphi\leq0$, we can easily continue the local solution to all of $\mathbb{R}$ and obtain that $\phi(x)\to-\infty$ as $x\to\infty$, which fails to be in $H^1(\mathbb{R})$. Thus there is no solitary wave in this case.

Similarly we have that when $c=A_2$ there is no solitary wave. Therefore the theorem is proved.
\end{proof}

\subsection{The case when $\sigma\neq0$}

In this case we can rewrite \eqref{eqnphi} as
\begin{equation}\label{eqnphi_sigmanot0}
 \LC (\varphi-{c\over\sigma})^2 \RC_{xx}=\varphi_x^2-{2(c+A)\over\sigma}\varphi
+{3\over\sigma}\varphi^2-{1\over\sigma}+{c^2\over \sigma(c-\varphi)^2},
\quad \hbox{in } \mathcal{D}'(\mathbb{R}).
\end{equation}

The following lemma deals with the regularity of the solitary waves.
The idea is inspired by the study of the traveling waves of
Camassa-Holm equation \cite{Le}.
\begin{lemma}\label{lm_smoothphi}
 Let $\sigma\neq0$ and $(\varphi,\eta)$ be a solitary wave of \eqref{genCH2_uandeta}.
Then
\begin{equation}\label{regsol}
 \LC \varphi-{c\over\sigma} \RC^k \in C^j\LC\mathbb{R}\backslash \varphi^{-1}(c/\sigma) \RC, \quad \hbox{for }\ k\geq 2^j.
\end{equation}
Therefore
\begin{equation}\label{regsol_1}
 \varphi\in C^\infty\LC \mathbb{R}\backslash \varphi^{-1}({c/\sigma}) \RC.
\end{equation}
\end{lemma}
\begin{proof}
 From Proposition \ref{prop_sol} we know that $c\neq0$ and $\varphi\neq c$ and thus $\varphi$
satisfies \eqref{eqnphi_sigmanot0}. Let $v=\varphi-{c\over\sigma}$
and denote
\begin{equation*}
 r(v)={3\over\sigma}\LC v+c\over\sigma \RC^2-{2(c+A)\over\sigma}\LC v+c\over\sigma \RC
-{1\over\sigma}.
\end{equation*}
So $r(v)$ is a polynomial in $v$. From the fact that
$\varphi-c\neq0$ we know that
\begin{equation}\label{cond_v}
{\sigma-1\over\sigma}c - v\neq0.
\end{equation}
Then $v$ satisfies
\begin{equation*}
 (v^2)_{xx}=v_x^2+r(v)+{c^2\over\sigma}\LC {\sigma-1\over\sigma}c - v \RC^{-2}.
\end{equation*}
From the assumption we know that $(v^2)_{xx}\in L^1_{{loc}}(\mathbb{R})$. Hence
$(v^2)_x$ is absolutely continuous and hence
\begin{equation*}
 v^2\in C^1(\mathbb{R}), \quad \hbox{and then }\ v\in C^1\LC \mathbb{R}\backslash v^{-1}(0) \RC.
\end{equation*}
So from \eqref{cond_v} and that $v+{c\over\sigma}\in
H^1(\mathbb{R})\subset C(\mathbb{R})$ we know
\begin{equation*}
 \LC {\sigma-1\over\sigma}c - v \RC^{-2}\in C(\mathbb{R})\cap C^1\LC\mathbb{R}\backslash v^{-1}(0) \RC.
\end{equation*}
Moreover,
\begin{align}
 (v^k)_{xx}&=(kv^{k-1}v_x)_x={k\over2}\LC v^{k-2}(v^2)_x \RC_x \nonumber\\
&=k(k-2)v^{k-2}v^2_x+{k\over2}v^{k-2}(v^2)_{xx} \nonumber\\
&=k(k-2)v^{k-2}v^2_x+{k\over2}v^{k-2}\LB v_x^2+r(v)+{c^2\over\sigma}
\LC {\sigma-1\over\sigma}c - v \RC^{-2} \RB \nonumber\\
&=k\LC k-{3\over2} \RC v^{k-2}v^2_x+{k\over2}v^{k-2}r(v)+{kc^2\over 2\sigma}v^{k-2}
\LC {\sigma-1\over\sigma}c - v \RC^{-2}.\label{v^2_xx}
\end{align}
For $k=3$, the right-hand side of \eqref{v^2_xx} is in $L^1_{loc}(\mathbb{R})$. Thus we deduce
that
\begin{equation*}
 v^3\in C^1(\mathbb{R}).
\end{equation*}
For $k\geq 4$ we see that \eqref{v^2_xx} implies
\begin{equation*}
 (v^k)_{xx}={k\over4}\LC k-{3\over2} \RC v^{k-4}\LB (v^2)_x \RB^2+{k\over2}v^{k-2}r(v)
+{kc^2\over 2\sigma}v^{k-2}\LC {\sigma-1\over\sigma}c - v \RC^{-2}
\in C(\mathbb{R}).
\end{equation*}
Therefore $v^k\in C^2(\mathbb{R})$ for $k\geq4$.

For $k\geq8$ we know from the above that
\begin{equation*}
 v^4, v^{k-4}, v^{k-2}, v^{k-2}r(v)\in C^2(\mathbb{R}),\ \hbox{and }\  v^{k-2}\LC {\sigma-1\over\sigma}c - v \RC^{-2}
 \in C^2\LC\mathbb{R}\backslash v^{-1}(0) \RC.
\end{equation*}
Moreover we have
\begin{equation*}
 v^{k-2}v^2_x={1\over4}(v^4)_x{1\over k-4}(v^{k-4})_x\in C^1(\mathbb{R}).
\end{equation*}
Hence from \eqref{v^2_xx} we conclude that
\begin{equation*}
 v^k\in C^3\LC\mathbb{R}\backslash v^{-1}(0) \RC, \quad k\geq8.
\end{equation*}

Applying the same argument to higher values of $k$ we prove that
$v^k\in C^j\LC\mathbb{R}\backslash v^{-1}(0) \RC$ for $k\geq 2^j$,
and hence \eqref{regsol}.
\end{proof}

Denote $\bar{x}=\min\{ x:\ \varphi(x)=c/\sigma \}$ (if $\varphi\neq c/\sigma$ for all $x$
then let $\bar{x}=+\infty$), then $\bar{x}\leq+\infty$. From
Lemma \ref{lm_smoothphi}, a solitary wave $\varphi$ is smooth on $(-\infty, \bar{x})$ and
hence \eqref{eqnphi} holds pointwise on $(-\infty, \bar{x})$. Therefore we may multiply by
$\varphi_x$ and integrate on $(-\infty,x]$ for $x<\bar{x}$ to get
\begin{equation}\label{eqnphi_x}
 \varphi^2_x={\varphi^2(c-\varphi-A_1)(c-\varphi-A_2) \over (c-\varphi)(c-\sigma\varphi)}
 := F(\varphi),
\end{equation}
where $A_1$ and $A_2$ are defined in \eqref{def_A_12}.

Applying the similar arguments as introduced in \cite{Le} we make the following conclusions.

1. When $\varphi$ approaches a simple zero $m=c-A_1$ or $m=c-A_2$ of $F(\varphi)$ so that $F(m)=0$ and $F'(m)\neq0$. The the solution
$\varphi$ of \eqref{eqnphi_x} satisfies
\begin{equation*}
 \varphi_x^2=(\varphi-m)F'(m)+O((\varphi-m)^2) \quad \hbox{as}\quad \varphi\to m,
\end{equation*}
where $f=O(g)$ as $x\to a$ means that $|f(x)/g(x)|$ is bounded in some interval
$[a-\epsilon, a+\epsilon]$ with $\epsilon>0$. Hence
\begin{equation}\label{simplezero}
 \varphi(x)=m+{1\over4}(x-x_0)^2F'(m)+O((x-x_0)^4) \quad \hbox{as}\quad x\to x_0,
\end{equation}
where $\varphi(x_0)=m$.

2. If $F(\varphi)$ has a double zero at $\varphi=0$, so that $F'(0)=0, F''(0)>0$, then
\begin{equation*}
 \varphi_x^2=\varphi^2F''(0)+O(\varphi^3) \quad \hbox{as}\quad \varphi\to 0.
\end{equation*}
We get
\begin{equation}\label{doublezero}
 \varphi(x)\sim \alpha\exp\LC -x\sqrt{F''(0)} \RC \quad \hbox{as}\quad x\to \infty
\end{equation}
for some constant $\alpha$. thus $\varphi\to 0$ exponentially as $x\to\infty$.

3. If $\varphi$ approaches a simple pole $\varphi(x_0)=c/\sigma$ of $F(\varphi)$ (when
$\sigma\neq1$). Then
\begin{align}\displaystyle
 &\varphi(x)-{c\over\sigma}=\beta|x-x_0|^{2/3} + O((x-x_0)^{4/3}) \quad \hbox{as} \quad x\to x_0,
\label{polephi}\\
&\varphi_x=\left\{\begin{array}{l}
                   {2\over3}\beta|x-x_0|^{-1/3}+O((x-x_0)^{1/3})\quad \hbox{as }\
                   x\downarrow x_0,\\
                   -{2\over3}\beta|x-x_0|^{-1/3}+O((x-x_0)^{1/3})\  \hbox{as }\
                   x\uparrow x_0,
                  \end{array}
           \right. \label{polephi_x}
\end{align}
for some constant $\beta$. In particular, when $F(\varphi)$ has a pole, the solution $\varphi$
has a cusp.

4. Peaked solitary waves occur when $\varphi$ suddenly changes direction: $\varphi_x\mapsto -\varphi_x$
according to \eqref{eqnphi_x}.

\vspace{.1in}

Now we give the following theorem on the existence of solitary waves of \eqref{genCH2_uandeta} for $\sigma\neq0$.

\begin{theorem}\label{thm_solexist}
For $\sigma\neq0$, we have
 \begin{enumerate}
  \item[(1)] If $0<\sigma\leq1$, a solitary wave  $(\varphi, \eta)$ of \eqref{genCH2_uandeta} exists if and only if condition \eqref{cond_c} holds.

If $c>A_1$ then $\varphi>0$ and
$\max_{x\in\mathbb{R}}\varphi(x)=c-A_1$. 
If $c<A_2$ then $\varphi<0$ and
$\min_{x\in\mathbb{R}}\varphi(x)=c-A_2$.

\vspace{.07in}

  \item[(2)] If $\sigma<0$, then
\begin{itemize}
\item if $c>A_1$ then there is a smooth solitary wave $\varphi>0$ with $\max_{x\in\mathbb{R}}\varphi(x)=c-A_1$,
and an anticusped solitary wave (the solution profile has a cusp pointing downward) $\varphi<0$ with $\min_{x\in\mathbb{R}}\varphi(x)={c/\sigma}$;
\item if $c<A_2$ then there is a smooth solitary wave $\varphi<0$ with $\min_{x\in\mathbb{R}}\varphi(x)=c-A_2$,
and a cusped solitary wave $\varphi>0$ with $\max_{x\in\mathbb{R}}\varphi(x)={c/\sigma}$;
\item if $c=A_1$ then there is an anticusped solitary wave $\varphi<0$ with $\min_{x\in\mathbb{R}}\varphi(x)={c/\sigma}$;
\item if $c=A_2$ then there is a cusped solitary wave $\varphi>0$ with $\max_{x\in\mathbb{R}}\varphi(x)={c/\sigma}$.
\end{itemize}

\vspace{.07in}

  \item[(3)] If $\sigma>1$, a solitary wave exists if and only if $c$ satisfies \eqref{cond_c}. If $c>A_1$ then $\varphi>0$. If $c<A_2$ then $\varphi<0$. Moreover,
\begin{itemize}
 \item If $A_1<c<{\sigma\over\sigma-1}A_1$, then the solitary waves are smooth and unique up to translation with
$\max_{x\in\mathbb{R}}\varphi(x)=c-A_1$;
 \item If $c={\sigma\over\sigma-1}A_1$ then the solitary wave is peaked
with $ \max_{x\in\mathbb{R}}\varphi(x)=c-A_1={c/\sigma}$;
 \item If $c>{\sigma\over\sigma-1}A_1$ then the solitary waves are cusped
with $\max_{x\in\mathbb{R}}\varphi(x)={c/\sigma}$;
 \item If ${\sigma\over\sigma-1}A_2<c<A_2$ then the solitary waves are smooth and unique up to translation with
$\min_{x\in\mathbb{R}}\varphi(x)=c-A_2$;
 \item If $c={\sigma\over\sigma-1}A_2$ then the solitary wave is antipeaked (the solution profile has a peak pointing downward)
with $\min_{x\in\mathbb{R}}\varphi(x)=c-A_2={c/\sigma}$;
 \item If $c<{\sigma\over\sigma-1}A_2$ then the solitary waves are anticusped
with $\min_{x\in\mathbb{R}}\varphi(x)={c/\sigma}$.
\end{itemize}
\end{enumerate}

Moreover, each kind of the above solitary waves is unique and even up to
translations. When $c>A_1$ or $c<A_2$, all solitary waves decay exponentially to zero at infinity.
\end{theorem}
\begin{proof}
 First from \eqref{eqnphi_x} and the decay of $\varphi(x)$ at infinity we know that a necessary condition for the existence
of solitary wave is that $c\geq A_1$ or $c\leq A_2$.

If $c=A_1$ then \eqref{eqnphi_x} becomes
\begin{equation}\label{eqnphi_x-A_1}
 \varphi^2_x={-\varphi^3(A_1-A_2-\varphi) \over (A_1-\varphi)(A_1-\sigma\varphi)}:=F_1(\varphi).
\end{equation}
Hence we see that $\varphi(x)<0$ near $-\infty$. Similarly as in the proof of Theorem \ref{thm_exist_sigma=0} , we can find some $x_0$ sufficiently large negative with $\varphi(x_0)=-\epsilon<0$ and $\varphi_x(x_0)<0$, and we can construct a unique local solution $\varphi(x)$ on $[x_0-L, x_0+L]$ for some $L>0$.

If $\sigma<0$, we see that
${1\over A_1-\sigma\varphi}$ is decreasing when $\varphi<0$. Together with \eqref{eqn_decreasing} we see that  $F_1(\varphi)$ decreases for $\varphi<0$.
Because $\varphi_x(x_0)<0$, $\varphi$ decreases near $x_0$, so
$F_1(\varphi)$ increases near $x_0$. Hence from
\eqref{eqnphi_x-A_1}, $\varphi_x$ decreases near $x_0$, and then
$\varphi$ and $\varphi_x$ both decreases on $[x_0-L, x_0+L]$. Since
$\sqrt{F_1(\varphi)}$ is locally Lipschitz in $\varphi$ for
$A_1/\sigma<\varphi\leq0$, we can easily continue the local solution
to $(-\infty, x_0-L]$ with $\varphi(x)\to0$ as $x\to-\infty$. As for
$x\geq x_0+L$, we can solve the initial valued problem
\begin{equation*}
\left\{\begin{array}{l}
        \psi_x=-\sqrt{F_1(\psi)},\\
        \psi(x_0+L)=\varphi(x_0+L)
       \end{array}
\right.
\end{equation*}
all the way until $\psi=A_1/\sigma$, which is a simple pole of
$F_1(\psi)$. From \eqref{polephi} and \eqref{polephi_x} we know that
we can construct an anticusped solution with a cusp singularity at
$\varphi=A_1/\sigma=c/\sigma$.

If $\sigma>0$, a direct computation shows that
\begin{equation*}
 F'_1(\varphi)<0, \quad \hbox{for }\quad \varphi<0.
\end{equation*}
Therefore the same argument indicates that $\varphi(x)\to-\infty$ as $x\to+\infty$, which fails to be in $H^1(\mathbb{R})$.
Hence in this case there is no solitary wave.

Similarly we conclude that when $c=A_2$, there is no solitary wave when $\sigma>0$. When $\sigma<0$, there is a solitary
wave with a cusp of height $c/\sigma$.

\vspace{.08in}

Now we consider $c>A_1$ or $c<A_2$. Again we will only look at
$c>A_1$. The other case $c<A_2$ can be handled in a very similar
way. From \eqref{eqnphi_x} we see that $\varphi$ can not oscillate
around zero near infinity. Let us consider the following two cases.

{\bf Case 1.} $\varphi(x)>0$ near $-\infty$. Then there is some
$x_0$ sufficiently large negative so that $\varphi(x_0)=\epsilon>0$,
with $\epsilon$ sufficiently small, and $\varphi_x(x_0)>0$.

(i) When $\sigma\leq1$, $\sqrt{F(\varphi)}$ is locally Lipschitz in
$\varphi$ for $0\leq\varphi\leq c-A_1$. Hence there is a local
solution to
\begin{equation*}
\left\{\begin{array}{l}
        \varphi_x=\sqrt{F(\varphi)},\\
        \varphi(x_0)=\epsilon
       \end{array}
\right.
\end{equation*}
on $[x_0-L, x_0+L]$ for some $L>0$. Therefore from
\eqref{simplezero} and \eqref{doublezero} we see that in this case
we can obtain a smooth solitary wave with maximum height
$\varphi=c-A_1$ and an exponential decay to zero at infinity
\begin{equation}\label{expdecay}
\varphi(x)=O\LC \exp\LC -{\sqrt{c^2+Ac-1} \over c}|x| \RC \RC \quad
\hbox{as} \quad |x|\to\infty.
\end{equation}

(ii) When $\sigma>1$, $\sqrt{F(\varphi)}$ is locally Lipschitz in
$\varphi$ for $0\leq\varphi< c/\sigma$. Thus if $c-A_1< c/\sigma$,
i.e., $c<{\sigma\over\sigma-1}A_1$, it becomes the same as (i) and
hence we obtain smooth solitary waves with exponential decay.

If $c-A_1=c/\sigma$ then the smooth solution can be constructed
until $\varphi=c-A_1=c/\sigma$. However at $\varphi=c-A_1=c/\sigma$
it can make a sudden turn and so give rise to a peak. Since
$\varphi=0$ is still a double zero of $F(\varphi)$, we still have
the exponential decay here.

Lastly if $c-A_1>c/\sigma$, then $\varphi=c/\sigma$ becomes a pole
of $F(\varphi)$. Hence from \eqref{polephi} and \eqref{polephi_x} we
see that we obtain a solitary wave with a cusp at $\varphi=c/\sigma$
and decays exponentially.

{\bf Case 2.} $\varphi(x)<0$ near $-\infty$. In this case we are
solving
\begin{equation*}
\left\{\begin{array}{l}
        \varphi_x=-\sqrt{F(\varphi)},\\
        \varphi(x_0)=-\epsilon
       \end{array}
\right.
\end{equation*}
for some $x_0$ sufficiently large negative and $\epsilon>0$
sufficiently small.

When $\sigma>0$ we see that $F'(\varphi)<0$ for $\varphi<0$. Thus in
this case there is no solitary wave.

When $\sigma<0$, $\varphi=c/\sigma<0$ is a pole of $F(\varphi)$.
Hence from similar argument as before, we obtain an anticusped
solitary wave with $\min_{x\in\mathbb{R}}=c/\sigma$, which decays
exponentially.

Finally, from the standard ODE theory and the fact that the equation
\eqref{eqnphi} is invariant under the transformations $x\mapsto x+d$
for any constant $d$, and $x\mapsto-x$, we conclude that the
solitary waves obtained above are unique and even up to
translations.



\end{proof}

\vspace{.08in}

Though there is no explicit expression for $\varphi$, and so $\eta$
in view of \eqref{soleta_phi}, as in \cite{LZ}, the effects of the
traveling speed $c$ on the function $\varphi$ can be analyzed to
provide some general description of its profile. Similarly to the
case in \cite{LZ} we have
\begin{proposition}\label{prop_phi_c}
Let $c>A_1$ or $c<A_2$, and $\varphi$ is a smooth solitary wave of
\eqref{genCH2_uandeta} as obtained in Theorem \ref{thm_solexist}.
Then $\partial_c\varphi$ decays exponentially to zero at infinity
and has at most two zeros on $\mathbb{R}$. In particular, if $A_1<c<{2\over A}$, the $\partial_c\varphi$ has exactly two zeros on $\mathbb{R}$.
\end{proposition}
\begin{proof}
Again we only discuss the case $c>A_1$. The other case $c<A_2$ can
be handled in the same way.

Denote $\omega=\partial_c\varphi$. The exponential decay of $\omega$
can be inferred from \eqref{expdecay}. Since $\varphi$ is unique and
even up to translations, we may assume that $\varphi(0)=c-A_1$.
Hence $\omega(0)=1$ and $\omega$ is even. Assume $\omega(x_0)=0$ for some
$x_0>0$. Differentiating \eqref{eqnphi_x} with respect to $c$ and
evaluating at $x=x_0$ we get
\begin{align*}
2\varphi_x\omega_x&={\varphi^2\over c-\sigma\varphi}\LB 1+{1\over
(c-\varphi)^2}+{(c-\varphi)^2+A(c-\varphi)-1 \over
(c-\varphi)(c-\sigma\varphi)} \RB\\
&={\varphi^2\over c-\sigma\varphi}\LB 1+{1\over
(c-\varphi)^2}+{\varphi^2_x \over \varphi^2} \RB>0,
\end{align*}
since $c-\sigma\varphi>0$. Because $\varphi_x(x_0)<0$, we see from
the above inequality that $\omega_x(x_0)<0$. So $\omega$ is strictly
decreasing near $x_0$. It is then deduced from the continuity of
$\omega$ that it has at most two zeros on $\mathbb{R}$.

If $A_1<c<{2\over A}$, then from the decay estimate \eqref{expdecay}
we see that $\varphi$ decays faster at infinity as $c$ gets larger,
since
\begin{equation*}
\partial_c\LC {\sqrt{c^2+Ac-1}\over c} \RC={2-Ac \over
2c^2\sqrt{c^2+Ac-1}}>0.
\end{equation*}
Hence $\omega(x)<0$ at infinity. Therefore $\omega$ has at least two
zeros.  Thus combining the above argument we proved that $\omega(x)$
has exactly two zeros $\pm x_0$ in this case.
\end{proof}


Next we try to find an implicit formula for the peaked solitary waves. Let us consider only the case $c>A_1$. Then from Theorem \ref{thm_solexist} we know that peaked solitary waves exist only when $c={\sigma\over\sigma-1}A_1$. In this case we have
\begin{equation*}
\varphi^2_x={\varphi^2(c-A_2-\varphi)\over c-\varphi}.
\end{equation*}
Since $\varphi$ is positive, even with respect to some $x_0$ and decreasing on $(x_0,\infty)$, so for $x>x_0$ we have
\begin{equation*}
\varphi_x=-\varphi\sqrt{1-{A_2\over c-\varphi}}.
\end{equation*}
Hence from the separation of variables we get
\begin{equation*}
-(x-x_0)=\int^\varphi_{c-A_1}{dt\over t\sqrt{1-{A_2\over c-t}}}.
\end{equation*}
Let $w=1-{A_2\over c-t}$ the above becomes
\begin{align*}
-(x-x_0)&=\int^{1-{A_2\over c-\varphi}}_{1-{A_2\over A_1}} {-A_2\over \LB cw-(c-A_2) \RB(w-1)\sqrt{w}}\ dw\\
&=\int^{1-{A_2\over c-\varphi}}_{1-{A_2\over A_1}} {1\over \sqrt{w}}\LB {c\over cw-(c-A_2)}-{1\over w-1} \RB\ dw\\
&=\LC \sqrt{c\over c-A_2}\ln \LV {\sqrt{cw}-\sqrt{c-A_2}\over \sqrt{cw}+\sqrt{c-A_2}} \RV-\ln \LV {\sqrt{w}-1\over \sqrt{w}+1} \RV\RC \Big|^{1-{A_2\over c-\varphi}}_{1-{A_2\over A_1}}.
\end{align*}
Therefore we obtain an implicit formula for the peaked solitary waves.
\begin{equation}\label{peakonformula}
-|x-x_0|=\LC \sqrt{c\over c-A_2}\ln \LV {\sqrt{cw}-\sqrt{c-A_2}\over \sqrt{cw}+\sqrt{c-A_2}} \RV-\ln \LV {\sqrt{w}-1\over \sqrt{w}+1} \RV\RC \Big|^{1-{A_2\over c-\varphi}}_{w=1-{A_2\over A_1}}.
\end{equation}

Below is a figure of such a peaked solitary wave with $x_0=0$, $A=0$, and $c=\sigma=2$.

\vspace{.1in}

\includegraphics[width=5.0in]{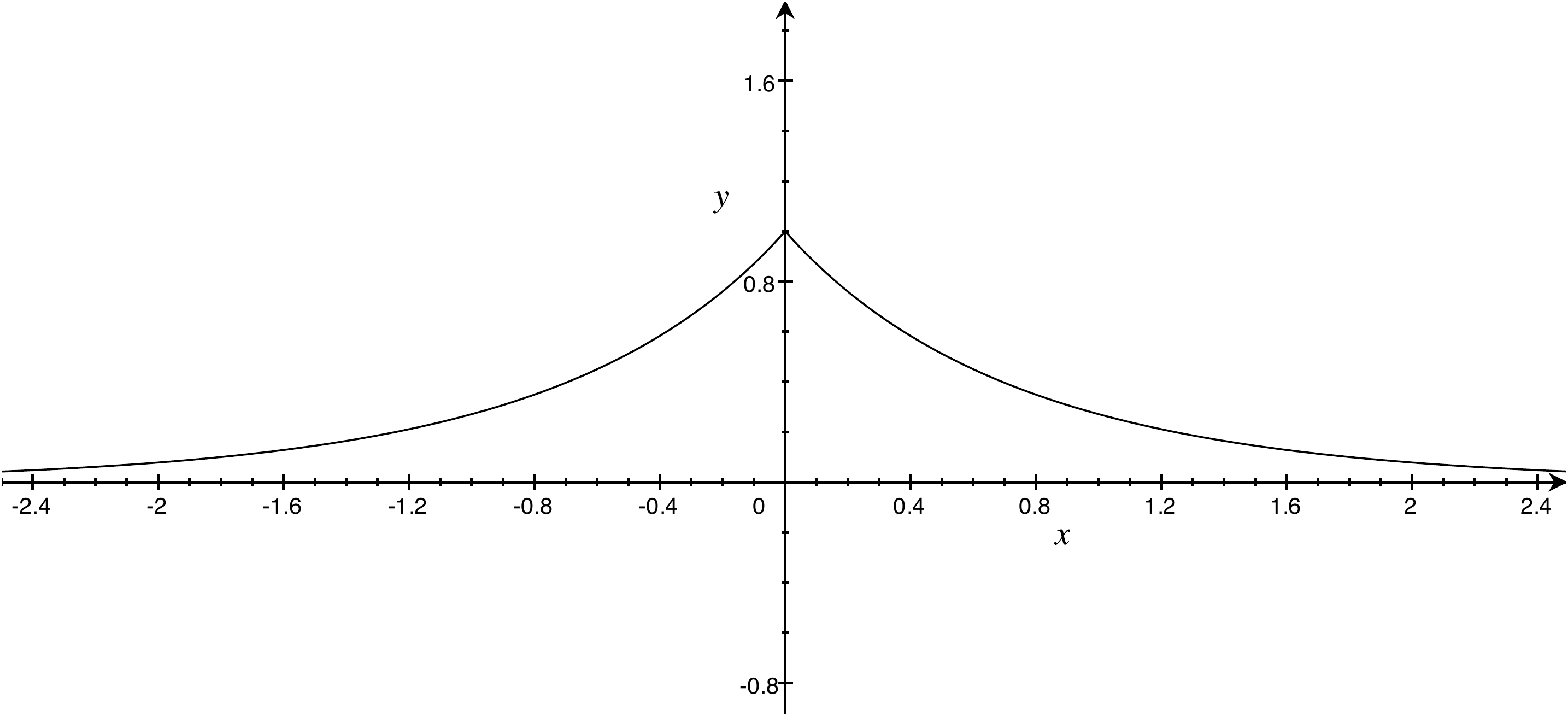}

\vspace{.1in}


\section{Stability}\label{sec_stab}


In this section, we want to discuss the stability of the smooth solitary waves of \eqref{genCH2_uandeta}. For fixed $c$ and
$\epsilon>0$, we define the ``$\epsilon$-tube'' of a solitary wave $\vec{\varphi}_c$ to be
\begin{equation}\label{def_tube}
 U_{\epsilon}=\{ \vec{u}\in X: \ \inf_{s\in\mathbb{R}}\|\vec{u}-\vec{\varphi}_c(\cdot-s)\|_{X}<\epsilon \}.
\end{equation}

According to Theorem \ref{thm_solexist}, the solitary waves for \eqref{genCH2_uandeta}
travel with speeds proportional to their maximal
heights. This consideration suggests that the appropriate
notion of stability for the solitary waves is
orbital stability: a wave starting close to a solitary
wave should stay close, as long as it exists, to some
translate of the solitary wave. The orbit of a solitary
wave is the set of all its translates.

Let us now discuss the appropriate notion of stability for the solitary waves of \eqref{genCH2_uandeta}.
\begin{definition}\label{def_stability}
 The solitary wave $\vec{\varphi}_c$ of \eqref{genCH2_uandeta} is stable in $X$ if for every $\epsilon>0$, there
exists a $\delta>0$ such that for any $\vec{u}_0\in U_{\delta}$, if $\vec{u}\in C([0,T); X)$ for some $0<T\leq \infty$
is a solution to \eqref{genCH2_uandeta} with $\vec{u}(0)=\vec{u}_0$, then $\vec{u}(t)\in U_{\epsilon}$ for all $t\in[0,T)$.
Otherwise the solitary wave $\vec{\varphi}_c$ is said to be unstable in $X$.
\end{definition}
As is discussed in \cite{CLi}, some solutions of \eqref{genCH2_uandeta} are defined globally in time
(e.g. for $0<\sigma<2$ and $\inf_{x\in\mathbb{R}}\eta_0>-1$, or the solitary waves constructed in Section
\ref{sec_solwaves}) while other waves break in finite time. Note that by stability we mean that even
if a solution which is initially close to a solitary wave blows up in a finite time, it will stay close
to some translate of the solitary wave up to the breaking time.

Our main theorem in this section is the following.

\begin{theorem}\label{thm_stability}
 Let $\sigma\leq 1$. All smooth solitary waves of \eqref{genCH2_uandeta} are stable.
\end{theorem}

First from Theorem \ref{thm_solexist} we know that smooth solitary waves exist only when $c>A_1$ or $c<A_2$. For
convenience we assume $c>A_1$.

The special case $\sigma=1$ is settled in \cite{LZ}. We will show that for general $\sigma< 1$, the problem can be analyzed using the method provided by Grillakis, Shatah, and Strauss \cite{GSS}.

Let us now make some functional analysis setup. Recalling the functionals $E$ and $F$ are well-defined on $X$, we may
compute their Fr\'{e}chet derivatives as follows
\begin{equation*}
 \left\{\begin{array}{l}
         E'_u=-u_{xx}+u\\
         E'_{\eta}=\eta,
        \end{array}
 \right. \quad
 \left\{\begin{array}{l}
         F'_u={3\over2}u^2-{\sigma\over2}u^2_x-\sigma uu_{xx}+\eta+{1\over2}\eta^2-Au\\
         F'_{\eta}=u+u\eta.
        \end{array}
 \right.
\end{equation*}
Using these notation we see that a solitary wave $\vec{\varphi}_c$ of \eqref{genCH2_uandeta} satisfies
\begin{equation}\label{eqnsolitary}
 cE'(\vec{\varphi}_c)-F'(\vec{\varphi}_c)=0.
\end{equation}
Denote
\begin{equation*}
 L_c=-\partial_x\LC (c-\sigma\varphi)\partial_x \RC-3\varphi+\sigma\varphi_{xx}+c+A.
\end{equation*}
Then the linearized operator $H_c: X\to X^*$ of $cE'-F'$ at $\vec{\varphi}_c$ can be computed as
\begin{equation*}
 H_c=cE''(\vec{\varphi}_c)-F''(\vec{\varphi}_c)
=\LC \begin{array}{cc}
      L_c & -(1+\eta)\\
      -(1+\eta) & c-\varphi
     \end{array}
  \RC.
\end{equation*}
Using \eqref{soleta_phi} we have
\begin{equation}\label{Hessian}
 H_c=\LC \begin{array}{cc}
      L_c & \displaystyle  -{c\over c-\varphi}\\\\
\displaystyle      -{c\over c-\varphi} & c-\varphi
     \end{array}
  \RC.
\end{equation}
We see easily that $H_c$ is self-adjoint and bounded from below, i.e., $H_c\geq aI$ for some constant $a$ and $I$ is the identity operator.

The next lemma states some spectral properties about $H_c$.
\begin{lemma}\label{lm_specH_c}
Let $c>A_1$ and $\vec{\varphi}_c$ be a smooth solitary wave of \eqref{genCH2_uandeta}. The spectrum of $H_c$
satisfies the following properties.
\begin{enumerate}
\item[(1)] The essential spectrum of $H_c$ is positive and bounded away from zero.
\item[(2)] The kernel of $H_c$ is spanned by $\partial_x\vec{\varphi}_c$.
\item[(3)] $H_c$ has exactly one negative simple eigenvalue $\lambda_1$ corresponding to eigenfunction $\vec{\chi}=(\chi, \mu)$.
\end{enumerate}
\end{lemma}
\begin{proof}
The proof is inspired  by Lemma 3.2 in \cite{LZ}. The details are as follows.
\begin{enumerate}
\item[(1)] Since $\varphi, \varphi_x$ and $\varphi_{xx}$ all decay exponentially at infinity, it follows from Weyl's essential spectrum theorem that the essential spectrum of $H_c$ is the same as that of its asymptotic operator $O_\infty$ as $|x|\to\infty$, where
\begin{equation*}
O_\infty=\LC \begin{array}{cc}
             -c\partial_{xx}+c+A & \ -1\\
             -1 & c
                       \end{array}
                \RC.
\end{equation*}
When $c>A_1$, we have $c^2+Ac-1>0$. Hence there is some constant $\delta=\delta(c,A)\in(0,1)$ such that
\begin{align*}
2|\psi\omega|&\leq 2(1-\delta)\sqrt{c(c+A)}|\psi\omega|\\
&\leq (1-\delta)\LB (c+A)\psi^2+c\omega^2 \RB
\end{align*}
for any $\vec{\psi}=(\psi,\omega)\in H^1(\mathbb{R})\times L^2(\mathbb{R})$. Therefore
\begin{align*}
(O_\infty\vec{\psi},\vec{\psi}^T)_{L^2\times L^2}&=\Int\ \LB c\psi_x^2+(c+A)\psi^2-2\psi\omega+c\omega^2 \RB\ dx\\
&\geq \Int\ \LB \delta(c+A)\psi^2+\delta c\omega^2 \RB\ dx\geq \delta c\|\vec{\psi}\|^2_{L^2\times L^2}.
\end{align*}
Hence $O_\infty$ is positive when $c>A_1$, and then the essential spectrum of $H_c$ is $[a_0,\infty)$ for some $a_0>0$ and there are finitely many eigenvalues located to the left of $a_0$.
\item[(2)] If $\vec{\psi}=(\psi,\omega)$ is an eigenfunction of $H_c$ corresponding to the eigenvalue zero, then
\begin{align*}
\displaystyle -\partial_x\LC (c-\sigma\varphi)\psi_x \RC+(-3\varphi+\sigma\varphi_{xx}+c+A)\psi-{c\over c-\varphi}\omega&=0,\\
\displaystyle -{c\over c-\varphi}\psi +(c-\varphi)\omega&=0.
\end{align*}
From the second equation we get $\omega={c\psi\over (c-\varphi)^2}$. Hence the first equation can be expressed as
a zero eigenvalue problem for $\mathcal{K}_c: H^1\to H^{-1}$:
\begin{equation}\label{eigenprob}
\mathcal{K}_c\psi:=-\partial_x\LC (c-\sigma\varphi)\psi_x \RC+\LC-3\varphi+\sigma\varphi_{xx}+c+A-{c^2\over (c-\psi)^3}\RC\psi=0.
\end{equation}

We now use the fact that $\varphi(x), \varphi_x(x), \varphi_{xx}(x) \to 0$ exponentially fast as $|x|\to\infty$ while
$c-\sigma\varphi$ is positive and bounded away from zero when $\varphi$ is smooth. Similar to \cite{CW4},
it follows that the spectral equation $\mathcal{K}_c\psi=0$ can be transformed by the Liouville substitution
\begin{equation*}
z=\int^x_0 {dy\over\sqrt{c-\sigma\varphi(y)}}, \quad \theta(z)=(c-\sigma\varphi(x))^{1/4}\psi(x),
\end{equation*}
into
\begin{equation*}
\mathcal{L}_c\theta(z):=\LC -\partial_z^2+q_c(z)+c+A-{1\over c} \RC\theta(z)=0,
\end{equation*}
where
\begin{equation*}
q_c(z)=-3\varphi(x)+{3\sigma\over4}\varphi_{xx}(x)-{c^2\over(c-\varphi(x))^3}+{1\over c}-{\sigma^2\varphi^2_x(x)\over 8(c-\sigma\varphi(x))}.
\end{equation*}
Since $q_c\to0$ exponentially as $|z|\to\infty$, we deduce that $\mathcal{L}_c: H^1(\mathbb{R})\to H^{-1}(\mathbb{R})$ is
self-adjoint with essential spectrum $[c+A-{1\over c},\infty)$. Because $c>A_1$, we know $c+A-{1\over c}>0$. We may have
finitely many eigenvalues of $\mathcal{L}_c$ located to the left of $c+A-{1\over  c}$. The $n$th eigenvalue (in increasing
order) has, up to a constant multiple, a unique eigenfunction with precisely $(n-1)$ zeros (see for example, \cite{DS} for details).  Thus the operator $\mathcal{K}_c$ has the same spectral properties.

Note that $\eta_x={c\varphi_x\over(c-\varphi)^2}$ and \eqref{soleqnfull} imply that $\mathcal{K}_c(\varphi_x)=0$.
Since $\varphi_x$ has exactly one zero. Therefore the zero eigenvalue of $\mathcal{K}_c$ is simple, and there is
exactly one negative eigenvalue while the rest of the spectrum is positive and bounded away from zero. Hence the
zero eigenvalue of $H_c$ is simple and the kernel is spanned by $\vec{\varphi}_x$.
\item[(3)] The operator $H_c$ is related to a quadratic form $Q_c(\vec{\psi})$ with $\vec{\psi}=(\psi,\omega)\in X$,
which is defined as the coefficient of $\epsilon^2$ in the Taylor's expansion of $cE(\vec{\varphi}_c+\epsilon\vec{\psi})
-F(\vec{\varphi}_c+\epsilon\vec{\psi})$ and is given by
\begin{align*}
Q_c(\vec{\psi})&={1\over2}\Int\ \Big[ (c-\sigma\varphi)\psi^2_x+(-3\varphi+\sigma\varphi_{xx}+c+A)\psi^2-{2c\over c-\varphi}\psi\omega\\
 &\quad \quad +(c-\varphi)\omega^2 \Big] \ dx\\
 &={1\over2}\Int\ \LB (c-\sigma\varphi)\psi^2_x+\LC-3\varphi+\sigma\varphi_{xx}+c+A-{c^2\over(c-\varphi)^3} \RC\psi^2\RB \ dx +\\
 &\quad\ {1\over2}\Int\ \LB (c-\varphi)\LC {c\over(c-\varphi)^2}\psi-\omega \RC^2 \RB \ dx\\
 &:=Q^{(1)}_c(\psi)+G(\vec{\psi}).
\end{align*}
Note that the quadratic form $Q^{(1)}_c(\psi)$ is related to the operator $\mathcal{K}_c$ and $G(\vec{\psi})$ is nonnegative.

Let $f$ be a nontrivial eigenfunction corresponding to the unique negative eigenvalue of $\mathcal{K}_c$. Then
\begin{equation*}
Q_c\LC f, {c\over(c-\varphi)^2}f \RC=Q^{(1)}_c(f)<0.
\end{equation*}
So $H_c$ has a negative eigenvalue, say, $\lambda_1<0$. Applying the min-max characterization of eigenvalues to $H_c$ yields
\begin{equation*}
\lambda_2=\max_{\vec{\psi}\in X}\min_{{\small\begin{array}{c}\vec{\omega}\in X\backslash\{0\} \\(\vec{\psi},\vec{\omega})=0\end{array}}} {Q_c(\vec{\omega})\over \|\omega\|^2_{X}}.
\end{equation*}
Choosing $\vec{\psi}=(f,0)$ leads to
\begin{equation*}
\lambda_2\geq \min_{{\small\begin{array}{c}\vec{\omega}\in X\backslash\{0\} \\((f,0),\vec{\omega})=0\end{array}}} {Q_c(\vec{\omega})\over \|\omega\|^2_{X}} = \min_{{\small\begin{array}{c}(g,h)\in X\backslash\{0\} \\((f,0),(g,h))=0\end{array}}} {Q^{(1)}_c(g)+G(g,h)\over \|(g,h)\|^2_{X}} \geq 0.
\end{equation*}
The last inequality is due to $Q^{(1)}_c(g)\geq 0$ for all $g$ such that $(g, f)_{H^1}=0$ and that $G(g,h)\geq0$. Therefore $\lambda_1$ is simple. Denote the corresponding eigenfunction by $\vec{\chi}=(\chi,\mu)$. From the result in (2) we see that $\lambda_2=0$ is also simple. This completes the proof of the lemma.
\end{enumerate}
\end{proof}

\begin{remark}
(i) Notice that the above lemma applies to all smooth solitary waves of \eqref{genCH2_uandeta} without the restriction that $\sigma\leq1$.

Under the assumption $c<A_2$, we can consider the operator
\begin{equation*}
 H_c=-\LC \begin{array}{cc}
      L_c & \displaystyle  -{c\over c-\varphi}\\\\
\displaystyle      -{c\over c-\varphi} & c-\varphi
     \end{array}
  \RC=
  \LC \begin{array}{cc}
      -L_c & \displaystyle  {c\over c-\varphi}\\\\
\displaystyle      {c\over c-\varphi} & -c+\varphi
     \end{array}
  \RC.
\end{equation*}
In this setting we have for smooth solitary wave that $c<c-A_2\leq\varphi(x)<0$ and $c-\sigma\varphi$
is bounded away from zero. By a similar argument, all properties of $H_c$ in Lemma \ref{lm_specH_c} are still valid.

(ii) We will apply the method of Grillakis-Shatah-Strauss \cite{GSS} to establish the stability of smooth solitary waves. However our problem does not exactly fall into the frame work there since the operator $J$ is not onto. But in fact the invertibility of $J$ is only needed to get instability and is not required for stability (see Sections 3 and 4 in \cite{GSS} for more details). Hence the argument in Section 3 of \cite{GSS} can still be used here.
\end{remark}

Let $\vec{\varphi}_c=(\varphi,\eta)$ be a solitary wave of \eqref{genCH2_uandeta}. Consider the
following scalar function
\begin{equation}\label{def_d(c)}
d(c)=\left\{\begin{array}{l}
        cE(\vec{\varphi}_c)-F(\vec{\varphi}_c)\quad \hbox{if}\quad c>A_1,\\
        F(\vec{\varphi}_c)-cE(\vec{\varphi}_c)\quad \hbox{if}\quad c<A_2.
                  \end{array}
         \right.
\end{equation}
The next lemma shows that for $\sigma\leq1$, , $c>A_1$ or $c<A_2$ and $\vec{\varphi}_c=(\varphi,\eta)\in X$
being a smooth solitary wave of \eqref{genCH2_uandeta},  $d(c)$ is convex in $c$.

\begin{lemma}\label{lm_signofd''}
Assume $\sigma\leq1$, $c>A_1$ or $c<A_2$ and $\vec{\varphi}_c=(\varphi,\eta)$ is a smooth solitary
wave of \eqref{genCH2_uandeta}. Then $d''(c)>0$.
\end{lemma}
\begin{proof}
Consider first $c>A_1$. Differentiating $d(c)$ with respect to $c$ and then applying \eqref{eqnsolitary} we obtain
\begin{equation}\label{compd'(c)}
d'(c)=\langle cE'(\vec{\varphi}_c)-F'(\vec{\varphi}_c), \partial_c\vec{\varphi}_c \rangle+E(\vec{\varphi}_c)=E(\vec{\varphi}_c).
\end{equation}
In view of the even symmetry of $\varphi$, it follows from \eqref{soleta_phi} and \eqref{eqnphi_x} that
\begin{align*}
d'(c)=E(\vec{\varphi}_c)&=\int^\infty_0 \LB \varphi^2_x+\varphi^2+{\varphi^2\over(c-\varphi)^2} \RB\ dx\\
&=\int^\infty_0 \varphi^2\LB {(c-\varphi-A_1)(c-\varphi-A_2)\over (c-\varphi)(c-\sigma\varphi)}+1+{1\over (c-\varphi)^2} \RB\ dx.
\end{align*}
Recall that $0<\varphi(x)\leq c-A_1$ and $\varphi'(x)<0$ on $(0,\infty)$ when $c>A_1$. We have from \eqref{eqnphi_x} that
\begin{align}
d'(c)=&-\int^\infty_0 \varphi\varphi_x\sqrt{{(c-\varphi)(c-\sigma\varphi) \over (c-\varphi-A_1)(c-\varphi-A_2)}}\cdot \nonumber\\
&\quad \ \  \Big[ {(c-\varphi-A_1)(c-\varphi-A_2)\over (c-\varphi)(c-\sigma\varphi)}+1+{1\over (c-\varphi)^2} \Big] \ dx. \label{explicitd'(c)}
\end{align}
Introducing a change of variable $y=c-\varphi(x)$ the above becomes
\begin{align*}
d'(c)=\int^c_{A_1}(c-y)\sqrt{{y\LB (1-\sigma)c+\sigma y \RB \over (y-A_1)(y-A_2)}}\LB {(y-A_1)(y-A_2) \over y\LB (1-\sigma)c+\sigma y \RB}+1+{1\over y^2} \RB \ dx.
\end{align*}
Differentiating the above with respect to $c$ we have
\begin{align*}
d''(c)&=0+\int^c_{A_1} \partial_c\LCB (c-y)\sqrt{{(y-A_1)(y-A_2) \over y\LB (1-\sigma)c+\sigma y \RB}} \RCB\ dy \\
&\quad\quad+ \int^c_{A_1} \partial_c\LCB (c-y)\sqrt{{y\LB (1-\sigma)c+\sigma y \RB \over (y-A_1)(y-A_2)}}\LC 1+{1\over y^2} \RC \RCB \ dy\\
&:= \int^C_{A_1}I_1(y)\ dy + \int^C_{A_1} I_2(y)\ dy.
\end{align*}
Since $\varphi$ is smooth, we have that $c-\sigma\varphi>0$. Hence $(1-\sigma)c+\sigma y=c-\sigma\varphi>0$. Moreover, we know that $A_2<0<A_1<y<c$. Further explicit computation shows that
\begin{align*}
I_1(y)&=\sqrt{{(y-A_1)(y-A_2) \over y\LB (1-\sigma)c+\sigma y \RB}}\cdot\LC {[(1-\sigma)c+\sigma y]+y \over 2[(1-\sigma)c+\sigma y]} \RC>0.
\end{align*}
Here we don't need to assume $\sigma\leq1$. If $\sigma\leq1$, we have
\begin{align*}
I_2(y)&={(y^2+1)\LB (1-\sigma)c+\sigma y+{1\over 2}(1-\sigma)(c-y) \RB\over y\sqrt{(y-A_1)(y-A_2)y[(1-\sigma)c+\sigma y]}}>0.
\end{align*}
Therefore $d''(c)>0$.

The other case $c<A_2$ can be handled in a very similar way and hence we omit it.
\end{proof}

The next lemma can be obtained by exactly the same proof as in \cite{GSS}.
\begin{lemma}\label{lm_implicit}
Let $\sigma\leq1$ and $\vec{\varphi}_c=(\varphi,\eta)$ be a solitary wave of \eqref{genCH2_uandeta}. There exist $\epsilon>0$ and a unique $C^1$ map $s: U_\epsilon \to \mathbb{R}$  such that for every $\vec{u}\in U_\epsilon$ and $r\in \mathbb{R}$
\begin{enumerate}
\item
\begin{equation*}
(\vec{u}(\cdot+s(\vec{u})), \vec{\varphi}_x)=0,
\end{equation*}
\item
\begin{equation*}
s(\vec{u}(\cdot+r))=s(\vec{u})-r.
\end{equation*}
\end{enumerate}
\end{lemma}

By the spectrum analysis in Lemma \ref{lm_specH_c} and the convexity of $d(c)$ in Lemma \ref{lm_signofd''}, we can follow exactly the same idea as in \cite{GSS} Theorem 3.3 to get
\begin{lemma}\label{lm_lowerbound}
Let the assumptions of Lemma \ref{lm_signofd''} hold. There exists a constant $k=k(c)>0$ such that
\begin{equation}\label{lowerbound}
\langle H_c(\vec{\psi}), \vec{\psi} \rangle \geq
k\|\vec{\psi}\|^2_{X},
\end{equation}
for all $\vec{\psi}\in X$ satisfying $(\vec{\varphi}_c,
\vec{\psi})=(\vec{\varphi}'_c, \vec{\psi})=0$.
\end{lemma}

The following lemma can be obtained directly from Lemma \ref{lm_implicit} and Lemma \ref{lm_lowerbound}.
\begin{lemma}\label{lm_gap}
Let the assumptions of Lemma \ref{lm_signofd''} hold. There exists an $\epsilon>0$ such that
\begin{equation*}
F(\vec{\varphi}_c)-F(\vec{u})\geq {k\over
4}\|\vec{u}(\cdot+s(\vec{u}))-\vec{\varphi}_c\|^2_{X},
\end{equation*}
for $\vec{u}\in U_\epsilon$ satisfying $E(\vec{u})=E(\vec{\varphi}_c)$.
\end{lemma}

\begin{proof}[Proof of Theorem \ref{thm_stability}]
In view of Lemma \ref{lm_specH_c} and Lemma \ref{lm_gap}, the result of theorem is then a direct consequence of Theorem 3.5 in \cite{GSS}.
\end{proof}

\vspace{.1in}


\section{Global solutions when $\sigma=0$}\label{sec_global}


In \cite{CLi}, the authors established a blow-up criterion for $\sigma\neq0$ (cf. Theorem 3.3
in \cite{CLi}). In fact, the restriction of $\sigma\neq0$ can be removed using the same argument
and hence we get
\begin{theorem}\label{thm_blowup1}
 Let $(u,\rho)$ be the solution of
\eqref{genCH2_uandrho} with initial data $(u_0, \rho_0-1)\in
H^s(\mathbb{R})\times H^{s-1}(\mathbb{R})$, $s>3/2$, and $T$ the
maximal time of existence. Then
\begin{equation}\label{blowupcriterion1}
 T<\infty\quad \Rightarrow \quad \int^T_0\|u_x(\tau)\|_{L^\infty}d\tau=\infty.
\end{equation}
\end{theorem}

The wave-breaking phenomena for system \eqref{genCH2_uandrho} when $\sigma\neq0$ was discussed
in details in \cite{CLi}. Here we show that when $\sigma=0$ the solutions constructed in Theorem
\ref{thm_localwellposedness} are global-in-time.
\begin{theorem}\label{thm_globalwellposedness}
 Let $\sigma=0$. If $(u_0, \rho_0-1)\in H^s\times H^{s-1}$, $s>3/2$, then there exists
 a unique solution $(u,\rho-1)$ of \eqref{genCH2_uandrho} in
 $C([0,\infty); H^s\times H^{s-1})\cap C^1([0,\infty); H^{s-1}\times H^{s-2})$
 with $(u,\rho)|_{t=0}=(u_0,\rho_0)$. Moreover, the solution depends
 continuously on the initial data and the Hamiltonian $H_1$
 is independent of the existence time.
\end{theorem}

As discussed in \cite{CLi}, system \eqref{genCH2_uandrho} has two associated characteristics $q$ and $\tilde{q}$
given by the following initial-value problems
\begin{equation}\label{trajectory_u}
\left\{\begin{array}{ll}
\displaystyle \frac{\partial q}{\partial t}=u(t,q), & 0<t<T,\\
q(0,x)=x, & x\in\mathbb{R},
\end{array}\right.
\end{equation}
\begin{equation}
\left\{\begin{array}{ll}
\displaystyle \frac{\partial \tilde{q}}{\partial t}=\sigma u(t,\tilde{q}), & 0<t<T,\\
\tilde{q}(0,x)=x, & x\in\mathbb{R},
\end{array}\right.
\end{equation}
where $u\in C^1([0,T),H^{s-1})$ is the first component of the
solution $(u,\rho)$ to \eqref{genCH2_uandrho} with initial data
$(u_0,\rho_0)\in H^s\times H^{s-1}$ with $s>3/2$ and $T>0$ is the
maximal time of existence. When $\sigma=0$, the second one $\tilde{q}$ becomes stationary. Thus we will perform the estimates along
the first characteristics $q$.

A direct calculation shows that for $t>0, x\in\mathbb{R}$
\[
\displaystyle q_{x}(t,x)=e^{\int^t_0u_x(\tau,
q(\tau,x))d\tau}>0.
\]
Hence $q(t,\cdot): \mathbb{R}\to\mathbb{R}$ is a diffeomorphism of the line
for each $t\in[0,T)$. Hence the $L^\infty$ norm of any function
$v(t,\cdot)\in L^\infty(\mathbb{R}), t\in[0,T)$ is preserved under $q(t,\cdot)$ with
$t\in[0,T)$, i.e.,
\begin{equation}\label{normpreserving}
\|v(t,\cdot)\|_{L^\infty(\mathbb{R})}=\|v(t,q(t,\cdot))\|_{L^\infty(\mathbb{R})},\quad
t\in[0,T).
\end{equation}
Similarly we have
\begin{align}
&\inf_{x\in\mathbb{R}}v(t,x)=\inf_{x\in\mathbb{R}}v(t,q(t,x)),\quad
t\in[0,T),\\
&\sup_{x\in\mathbb{R}}v(t,x)=\sup_{x\in\mathbb{R}}v(t,q(t,x)),\quad
t\in[0,T). \label{suppreserving}
\end{align}


\vspace{.15in}



When $\sigma=0$, we can rewrite system \eqref{genCH2_uandrho} as
\begin{equation}\label{genCH2_sigma=0}
\left\{\begin{array}{l}
u_t+\partial_xp\ast\left(-Au+\frac{3}{2}u^2+\frac12\rho^2\right)=0,\\
\rho_t+(\rho u)_x=0,
\end{array}\right.
\end{equation}
where $p(x)$ is defined in \eqref{def_p}

The following lemma is needed in carrying out the estimates along the ``extremal'' characteristics.
\begin{lemma} (\cite{CE}) \label{lm_trajectory}
 Let $T>0$ and $v\in C^1\LC [0,T);H^2(\mathbb{R}) \RC$. Then for every $t\in [0,T)$ there exists at least one point $\xi(t)\in \mathbb{R}$ with
\[
 m(t):=\inf_{x\in\mathbb{R}}\LB v_x(t,x) \RB = v_x\LC t,\xi(t) \RC.
\]
The function $m(t)$ is absolutely continuous on $(0,T)$ with
\[
 {dm(t)\over dt} = v_{tx}\LC t,\xi(t) \RC \quad \hbox{ a.e. on } (0,T).
\]
\end{lemma}

To prove Theorem \ref{thm_globalwellposedness} of global well-posedness of solutions, we
need the following estimates for $u_x$.
\begin{lemma}\label{lm_estu_x}
 Let $\sigma=0$ and $(u,\rho)$ be the solution of
\eqref{genCH2_sigma=0} with initial data $(u_0, \rho_0-1)\in
H^s(\mathbb{R})\times H^{s-1}(\mathbb{R})$, $s>3/2$, and $T$ the
maximal time of existence. Then
\begin{align}
 &\sup_{x\in\mathbb{R}}u_x(t,x)\leq \sup_{x\in\mathbb{R}}u_{0,x}(x)
+{1\over2}\LC \sup_{x\in\mathbb{R}}\rho^2_0(x)+C_1^2 \RC t, \label{estsupu_x} \\
 &\inf_{x\in\mathbb{R}}u_x(t,x)\geq \inf_{x\in\mathbb{R}}u_{0,x}(x)
+{1\over2}\LC \inf_{x\in\mathbb{R}}\rho^2_0(x)-C_2^2 \RC t \label{estinfu_x},
\end{align}
where the constants above are defined as follows.
\begin{align}
&C_1=\sqrt{\frac{3+A^2}{2}}\|(u_0, \rho_0-1)\|_{H^1\times L^2},\label{defnC_1}\\
&C_2=\sqrt{2+C_1^2}.\label{defnC_2}
\end{align}
\end{lemma}

\begin{proof}
The local well-posedness theorem and a density argument implies that
it suffices to prove the desired estimates for $s\geq 3$. Thus we
take $s=3$ in the proof. Also we may assume that
\begin{equation}\label{nonzerocond}
u_0\not\equiv0.
\end{equation}
Otherwise the results become trivial. Since now $s\geq3$, we have $u\in C^1_0(\mathbb{R})$. 
Therefore
\begin{equation}\label{signcond}
 \inf_{x\in\mathbb{R}}u_x(t,x)\leq0, \quad \sup_{x\in\mathbb{R}}u_x(t,x)\geq0, \quad t\in[0,T).
\end{equation}

Differentiating the first equation of \eqref{genCH2_sigma=0} with respect to $x$ and using
the identity $-\partial^2_xp\ast f=f-p\ast f$ we obtain
\begin{equation}\label{diffeqn}
u_{tx}=\frac12\rho^2+\frac{3}{2}u^2+A\partial^2_xp\ast
u-p\ast\LC \frac{3}{2}u^2+\frac12\rho^2\RC.
\end{equation}

Using Lemma \ref{lm_trajectory} and the fact that
\begin{equation*}
 \sup_{x\in\mathbb{R}}\LB v_x(t,x) \RB=-\inf_{x\in\mathbb{R}}\LB -v_x(t,x) \RB,
\end{equation*}
we can consider $\bar{m}(t)$ and $\bar{\xi}(t)$ as follows
\begin{equation}\label{defnbarm}
 \bar{m}(t):= u_x\LC t,\bar{\xi}(t) \RC=\sup_{x\in\mathbb{R}}\LC u_x(t,x) \RC, \quad t\in
 [0,T).
\end{equation}
Hence
\begin{equation}\label{zerou_xxeta}
 u_{xx}\LC t, \bar{\xi}(t) \RC=0, \quad \hbox{a.e.}\quad t\in[0,T).
\end{equation}
Take the trajectory $q(t,x)$ defined in \eqref{trajectory_u}. Then we know that
$q(t,\cdot): \mathbb{R}\to\mathbb{R}$ is a diffeomorphism for every $t\in[0,T)$.
Therefore there exists $x_1(t)\in \mathbb{R}$ such that
\begin{equation}\label{choiceofx_1}
 q\LC t,x_1(t) \RC = \bar{\xi}(t) \quad t\in[0,T).
\end{equation}
Now let
\begin{equation}\label{defnbarzeta}
\bar{\zeta}(t)=\rho(t,q(t,x_1)),\quad
 t\in[0,T).
\end{equation}
Therefore along this trajectory $q(t,x_1)$ equation \eqref{diffeqn} and the
second equation of \eqref{genCH2_sigma=0} become
\begin{align}\label{blowupeqn}
&\bar{m}'(t)=\frac12\bar{\zeta}^2+f(t,q(t,x_1)),\nonumber\\
&\bar{\zeta}'(t)=-\bar{\zeta} \bar{m},
\end{align}
for $t\in[0,T)$, where $'$ denotes the derivative with respect to
$t$ and $f(t,q(t,x))$ is given by
\begin{equation}\label{defnf}
f=\frac{3}{2}u^2+A\partial^2_xp\ast
u-p\ast\LC \frac{3}{2}u^2+\frac12\rho^2\RC.
\end{equation}

We first derive the upper and lower bounds for $f$ for later use in
getting the wave-breaking result. Using that $\partial^2_x p\ast
u=p_x\ast u_x$, we have
\begin{align*}
f&=\frac{3}{2}u^2+Ap_x\ast
u_x-{3\over2}p\ast u^2-\frac12p\ast1
-p\ast(\rho-1)-\frac12p\ast(\rho-1)^2\\
&\leq\frac{3}{2}u^2+A|p_x\ast
u_x|-\frac12+|p\ast(\rho-1)|.
\end{align*}
Since
\begin{align}
A|p_x\ast u_x|&\leq
A\|p_x\|_{L^2}\|u_x\|_{L^2}=\frac12A\|u_x\|_{L^2}\leq\frac14+\frac14A^2\|u_x\|^2_{L^2},\label{control_convpu}\\
|p\ast(\rho-1)|&\leq
\|p\|_{L^2}\|\rho-1\|_{L^2}=\frac12\|\rho-1\|_{L^2}\leq
\frac14+\frac14\|\rho-1\|^2_{L^2},\label{control_convrho}\\
u^2 & \leq \frac{1}{2}\int_{\mathbb{R}}(u^2+u^2_x)\ dx,\label{control_u^2}
\end{align}
we obtain the upper bound of $f$
\begin{align}\nonumber
f&\leq
\frac14\|\rho-1\|^2_{L^2}+\frac{3}{4}\|u\|^2_{L^2}+
\frac{3+A^2}{4}\|u_x\|_{L^2}\\
&\leq \frac{3+A^2}{4}\|(u_0, \rho_0-1)\|^2_{H^1\times L^2}= \frac12C^2_1.\label{upperboundf}
\end{align}

Now we turn to the lower bound of $f$.
Similar as before, we get
\begin{align}
-f&\leq A|p_x\ast
u_x|+{3\over2}p\ast u^2+\frac12+|p\ast(\rho-1)|+\frac12p\ast(\rho-1)^2.\nonumber\\
&\leq
1+\frac{A^2}{4}\|u_x\|^2_{L^2}+\frac{3}{4}\|u\|^2_{L^2}+\frac12\|\rho-1\|^2_{L^2} \nonumber\\
&\leq 1+\frac{3+A^2}{4}\|(u_0, \rho_0-1)\|^2_{H^1\times L^2}= \frac12C^2_2, \label{lowerboundf}
\end{align}
where we have used the inequality
\begin{equation*}
 p\ast g^2\leq {1\over2}\|g^2\|_{L^1}={1\over2}\|g\|^2_{L^2}.
\end{equation*}

Combining \eqref{upperboundf} and \eqref{lowerboundf} we obtain
\begin{equation}\label{boundf}
 |f|\leq 1+\frac{3+A^2}{4}\|(u_0, \rho_0-1)\|^2_{H^1\times L^2}.
\end{equation}

From \eqref{signcond} we know $\bar{m}(t)\geq0$ for $t\in[0,T)$. From the second equation
of \eqref{blowupeqn} we obtain that
\begin{equation}\label{solnbarzeta}
\bar{\zeta}(t)=\bar{\zeta}(0) e^{-\int_0^t\bar{m}(\tau)d\tau}.
\end{equation}
Hence
\begin{align*}
 |\rho(t,q(t,x_1))|=|\bar{\zeta}(t)|\leq |\bar{\zeta}(0)|.
\end{align*}
Therefore we have
\begin{align*}
 \bar{m}'(t)={1\over2}\bar{\zeta}^2(t)+f\leq {1\over2}\bar{\zeta}^2(0)+{1\over2}C^2_1
\leq {1\over2}\LC \sup_{x\in\mathbb{R}}\rho^2_0(x)+C^2_1 \RC.
\end{align*}
Integrating the above from over $[0,t]$ we prove \eqref{estsupu_x}.

To obtain a lower bound for $\inf_{x\in\mathbb{R}}u_x(t,x)$, we use the similar idea.
Consider the functions $m(t)$ and $\xi(t)$ as in Lemma
\ref{lm_trajectory}
\begin{equation}\label{defnm}
 m(t):= u_x\LC t,\xi(t) \RC=\inf_{x\in\mathbb{R}}\LC u_x(t,x) \RC, \quad t\in [0,T).
\end{equation}
Hence
\begin{equation}\label{zerou_xx}
 u_{xx}\LC t,\xi(t) \RC = 0 \quad \hbox{ a.e. } t\in [0,T).
\end{equation}
Again take the characteristics $q(t,x)$ defined in \eqref{trajectory_u}
and choose $x_2(t)\in \mathbb{R}$ such that
\begin{equation}\label{choiceofx_2}
 q\LC t,x_2(t) \RC = \xi(t) \quad t\in[0,T).
\end{equation}
Let
\begin{equation}\label{defnzeta}
 \zeta(t)=\rho\LC t,q(t,x_2) \RC, \quad t\in[0,T).
\end{equation}
Hence along this trajectory $q(t,x_2)$ equation \eqref{diffeqn} and the
second equation of \eqref{genCH2_sigma=0} become
\begin{align}\label{blowupeqninf}
&m'(t)=\frac12{\zeta}^2+f(t,q(t,x_2)),\nonumber\\
&{\zeta}'(t)=-{\zeta} {m}.
\end{align}

Since $m(t)\geq0$, we have from the second equation of the above that
\begin{align*}
 |\rho(t,q(t,x_2))|=|{\zeta}(t)|\geq |{\zeta}(0)|.
\end{align*}
Then
\begin{align*}
 {m}'(t)\geq {1\over2}{\zeta}^2(0)-{1\over2}C^2_2
\geq {1\over2}\LC \inf_{x\in\mathbb{R}}\rho^2_0(x)-C^2_2 \RC.
\end{align*}
Integrating the above from over $[0,t]$ we obtain \eqref{estinfu_x}. This completes the proof of Lemma  Lemma \ref{lm_estu_x}.
\end{proof}

\vspace{.1in}

\begin{proof}[Proof of Theorem \ref{thm_globalwellposedness}]
Combining Lemma \ref{lm_estu_x} and Theorem \ref{thm_blowup1} we easily see that the local
solution obtained in Theorem \ref{thm_localwellposedness} can be extended to all of the interval $[0,\infty)$.
\end{proof}


\medskip
\medskip




\end{document}